\documentclass[a4paper,12pt]{amsart}

\usepackage{amsmath}
\usepackage{mathtools}
\usepackage{amssymb}

\usepackage{enumitem}

\DeclareMathOperator{\GL}{GL}

\DeclareMathOperator{\SL}{SL}

\DeclareMathOperator{\Gal}{Gal}
\DeclareMathOperator{\rank}{rank}

\DeclareMathOperator{\im}{im}

\DeclareMathOperator{\Sel}{Sel}

\DeclareMathOperator{\D}{\mathcal{D}}
\title[Integral points on the congruent number curves]{The average number of integral points on the congruent number curves}
\author{Stephanie Chan}
\email{ytchan@umich.edu}
\address{Department of Mathematics, University of Michigan, 530 Church Street
Ann Arbor, MI 48109, USA}
\keywords{elliptic curve, quadratic twist, integral point}
\newcommand{\E}{\mathcal{E}}

\newcommand{\Q}{\mathbb{Q}}

\newcommand{\Z}{\mathbb{Z}}

\newcommand{\leg}[2]{\left(\frac{#1}{#2}\right)}

\usepackage{dsfont}

\newtheorem{theorem}{Theorem}[section]
\newtheorem{corollary}[theorem]{Corollary}
\newtheorem{lemma}[theorem]{Lemma}

\newtheorem{definition}[theorem]{Definition}

\newtheorem{conjecture}[theorem]{Conjecture}

\begin{document}
\begin{abstract}
We show that the total number of non-torsion integral points on the elliptic curves $\E_D:y^2=x^3-D^2x$, where $D$ ranges over positive squarefree integers less than $N$, is $O( N(\log N)^{-1/4+\epsilon})$. The proof involves a discriminant-lowering procedure on integral binary quartic forms and an application of Heath-Brown's method on estimating the average size of the $2$-Selmer group of the curves in this family. 
\end{abstract}

\maketitle
%\tableofcontents

\section{Introduction}
Given an elliptic curve over $\Q$ with short Weierstrass model
\begin{equation}\label{eq:Emodel}
E: y^2=x^3+Ax+B,\ A,B\in\Z,\end{equation}
we study the quadratic twists of $E$, with the model
\begin{equation}\label{eq:twistmodel}
E_D:y^2=x^3+AD^2x+BD^3,\end{equation}
where $D$ denotes a positive squarefree integer. 
Consider the set of integral points
\[E_D(\Z)\coloneqq \left\{(x,y)\in\Z^2:y^2=x^3+AD^2x+BD^3\right\}.\]
It follows from a result of Mordell~\cite{MordellIntFinite} that $\#E_D(\Z)$ is always finite.

We are interested in the distribution of the number of integral points $\#E_D(\Z)$ in quadratic twist families, when $E_D$ are ordered according to the size of $D$. 
If $E(\Q)$ contains a $2$-torsion point, this point must have the form $(a,0)$ for some integer $a$ under the model~\eqref{eq:Emodel}, then $(aD,0)\in E_D(\Z)$ for all squarefree integers $D$. Therefore we call an integral point non-trivial if it is not a $2$-torsion point of $E_D(\Q)$. Define the set of non-trivial integral points on $E_D$ to be
\[ E_D^*(\Z)\coloneqq\{(x,y)\in E_D(\Z): y\neq 0\}.\]

Define
\begin{align*}
 & \D\coloneqq \{D\in\Z_{>0}:D\text{ squarefree}\},\\
 & \D_N\coloneqq \{D\in\D:D\leq N\}.
\end{align*}
Granville~\cite{Granvilletwists} conjectured that almost all curves within a quadratic twist family have no non-trivial integral point. We state the conjecture adapted to our model~\eqref{eq:twistmodel}.
\begin{conjecture}[Granville~\cite{Granvilletwists}]\label{conj:Granville}
 Fix $A,B\in\Z$ such that $4A^3+27B^2\neq 0$. Let $E_D: y^2=x^3+AD^2x+BD^3$, $D\in\mathcal{D}$. Then
 \begin{equation*}
 \#\{D\in\D_N: E_D^*(\Z)\neq \varnothing\} \sim C_{A,B}N ^{\frac{1}{2} },
 \end{equation*}
 where $C_{A,B}$ is a constant that depends only on $A,B$.
\end{conjecture}
We note that Granville's original conjecture considers a different model $Dy^2=f(x)$, where $f \in \Z[x]$ and $\deg f=3$.
When $f(x)=x^3+Ax+B$, any point $(x,y)\in\Z^2$ satisfying $Dy^2=f(x)$ corresponds to a point $(Dx,Dy)\in E_D(\Z)$, so there are fewer integral points using the model $Dy^2=f(x)$ when compared to our model~\eqref{eq:twistmodel}. The exponent $\frac{1}{2}$ stated in Conjecture~\ref{conj:Granville} replaces $\frac{1}{3}$ in the original conjecture because of this discrepancy. The exponent $\frac{1}{2}$ is suggested by some heuristics we gave in~\cite[p.~6677--6678]{Cintegral} for the family $y^2=x^3-D^2x$. 

In this direction, Matschke and Mudigonda~\cite{MM} handled the case when $f(x)$ is reducible, assuming the $abc$ conjecture.
\begin{theorem}[Matschke--Mudigonda~\cite{MM}]
 Assume that the $abc$ conjecture is true. Suppose $f(x)=x^3+Ax+B$, $A,B\in\Z$, such that $4A^3+27B^2\neq 0$ and $f(x)$ is reducible over $\Q$. Then
 \begin{equation*}
 \#\{D\in\D_N: Dy^2=f(x) \text{ for some }x,y\in\Z,\ y\neq 0\} \leq N^{\frac{2}{3} + o(1)}. 
 \end{equation*}
\end{theorem}

Our goal here is to gain progress towards Conjecture~\ref{conj:Granville} on a specific quadratic twist family.
We restrict our attention to the congruent number curve
$\E:y^2=x^3-x$,
and study its twists \[\E_D:y^2=x^3-D^2x.\] 
It is well known that the torsion subgroup of $\E_D(\Q)$ is $\{O,\ (0,0),\ (\pm D,0)\}\cong \Z/2\Z\times\Z/2\Z$ (see for example~\cite[Chapter~I, Proposition~17]{Koblitz}), where $O$ denotes the point at infinity. 

For this family $\E_D$, we can deduce from existing results that all moments of $\#\E_D(\Z)$ are finite. 
The $2$-Selmer groups of $\E_D$, which we denote by $\Sel_2(\E_D)$, is a finite group with order $2$ that is defined via local conditions and admits an injection $\E_D(\Q)/2\E_D(\Q)\hookrightarrow\Sel_2(\E_D)$ (see for example~\cite[Chapter~X]{Silverman}). In particular, the $2$-Selmer rank provides an upper bound to the rank $\rank(\E_D(\Q))$ of the Mordell--Weil group of $\E_D(\Q)$.
It is usually easier to compute the $2$-Selmer groups of elliptic curves with a torsion subgroup $\Z/2\Z\times\Z/2\Z$ over $\Q$, since then most of the work can be done over $\Q$. Heath-Brown~\cite[Theorem~1]{HBSelmer} computed all the moments of the size of the $2$-Selmer groups of $\E_D$. For any fixed positive integer $k$, he showed that
\begin{equation}\label{eq:HBmoments}
\lim_{N\rightarrow\infty}\frac{1}{\#\D_N}\sum_{D\in\D_N} (\#\Sel_2(\E_D))^{k}=c_k+o_k(1),
\end{equation}
where $c_k$ are explicit constants that can be bounded by $3^{k(k+1)}$.
Since the $2$-Selmer rank provides an upper bound to the the rank of $\E_D$,~\eqref{eq:HBmoments} implies that
\begin{equation}\label{eq:MWmoments}
\limsup_{N\rightarrow\infty}\frac{1}{\#\D_N}\sum_{D\in\D_N} 2^{k\cdot \rank \E_D(\Q)}\leq 3^{k(k+1)}.
\end{equation}
Lang~\cite[page 140]{LangDioph} conjectured that the number of integral points on a quasi-minimal Weierstrass equation of an elliptic curve $E$ should be bounded only in terms of $\rank E(\Q)$.
For the family $\E_D$, if follows from known results in this direction~\cite[Theorem~A]{SilvermanSie},~\cite[Theorem~0.7]{HS}, that there exists some absolute constant $C_1$, such that
\begin{equation}\label{eq:Cbound}
\# \E_D(\Z)\ll C_1^{\rank \E_D(\Q)}.
\end{equation} 
In~\cite{Cintegral}, we showed that $C_1$ in~\eqref{eq:Cbound} can be taken as $3.8$. 
Combining the upper bound~\eqref{eq:Cbound} and~\eqref{eq:MWmoments}, we can bound the $k$-th moment 
\begin{equation}\label{eq:intmoments}
\limsup_{N\rightarrow\infty}\frac{1}{\#\D_N}\sum_{D\in\D_N}(\#\E_D(\Z))^{k}\ll C_2^{k(k+1)},
\end{equation}
where $C_2$ is an absolute constant.

We will show that in fact the moments of $\#\E_D(\Z)$ should each tend to $0$. The following is our main result.
\begin{theorem}\label{theorem:main}
For any $\epsilon>0$, we have
\[
\sum_{D\in\D_N}\#\E_D^*(\Z)\ll N(\log N)^{-\frac{1}{4}+\epsilon}.
\]
\end{theorem}
This shows that the average size of $\#\E_D^*(\Z)$ tends to $0$ as $N$ tends to infinity, since $\#\D_N\sim \frac{6}{\pi^2}N$. 

Theorem~\ref{theorem:main} implies that 
\begin{equation}\label{eq:nointpt}
\#\{D\in\D_N:\E_D^*(\Z)\neq \varnothing\}\ll N(\log N)^{-\frac{1}{4}+\epsilon}.
\end{equation}
An application of Hölder's inequality using~\eqref{eq:intmoments} and~\eqref{eq:nointpt}, gives
\[\begin{split}
\sum_{D\in \D_N}(\#\E_D^*(\Z))^k
&\leq
\Big(\sum_{D\in \D_N}(\#\E_D^*(\Z))^{\frac{k}{\epsilon}}\Big)^{\epsilon}
\left(\#\{D\in\D_N:\E_D^*(\Z)\neq \varnothing\}\right)^{1-\epsilon}\\
&\ll NC_2^{(\frac{k}{\epsilon}+1)k}(\log N)^{(-\frac{1}{4}+\epsilon)(1-\epsilon)}.
\end{split}\]
Rescaling $\epsilon$ gives Corollary~\ref{cor:main}.

\begin{corollary}\label{cor:main}
For any $\epsilon>0$ and $k>0$, we have
\[
\sum_{D\in\D_N}(\#\E_D(\Z))^{k}\ll_{\epsilon,k} N(\log N)^{-\frac{1}{4}+\epsilon}.
\]
\end{corollary}

We now give an outline of the proof of Theorem~\ref{theorem:main}.
In Section~\ref{sec:quartics}, for each integral point $(x,y)\in \E_D(\Z)$, we use Mordell's correspondence~\cite[Chapter 25]{Mordell} to construct a corresponding integral binary quartic form $f$ that represents $1$ and has discriminant related to the discriminant of $E$. Then in Section~\ref{sec:reduce}, we show that by picking an auxiliary prime $p\mid D/\gcd(x,D)$, we can transform $f$ into an integral binary quartic form $F$ that represents $p$ and has discriminant lowered by a factor of $p^6$.
In Section~\ref{sec:Selmer}, we show that $\gcd(x,D)$ can be controlled by the image of $(x,y)$ in the $2$-Selmer group of $\E_D$ under the map
\[\E_D(\Z)\hookrightarrow \E_D(\Q)\twoheadrightarrow\E_D(\Q)/2\E_D(\Q)\hookrightarrow\Sel_2(\E_D).\] 
Then in Section~\ref{sec:genericSelmer} we extract some information about the distribution of $2$-Selmer elements from Heath-Brown's work~\cite{HBSelmer1,HBSelmer} to show that for almost all $D$, we are able to pick a prime $p$ of a suitable size. In particular, this $p$ is not too small, so that there are $o(N)$ many discriminants for the discriminant-lowered quartic $F$ to take. At the same time, this $p$ is not too large, so that each $\GL_2(\Z)$-equivalence class of $F$ can only be the image of finitely many integral points by applying bounds on the number of solutions to Thue inequalities.
In Section~\ref{section:nongen}, we use Hölder's theorem and~\eqref{eq:intmoments} to bound the contribution from the the exceptional curves to the number of integral points.
In Section~\ref{sec:count}, we proceed to count the set of those quartics $F$ that were discriminant-lowered by some suitable $p$. We make use of the fact that every integral binary quartic form is $\SL_2(\Z)$-equivalent to at least one reduced form with bounded seminvariants~\cite{Cremona}. Applying the syzygy satisfied by the seminvariants returns a set of integral points on twists of $\E$ with bounded height. Then Theorem~\ref{theorem:main} follows from an application of an upper bound by Le Boudec~\cite{LeBoudec}.

\section{Integer-matrix binary quartic forms}\label{sec:quartics}
We say that a binary quartic form is \emph{integer-matrix} if it has the form 
\[f(X,Y)=a_0X^4+4a_1X^3Y+6a_2X^2Y^2+4a_3XY^3+a_4Y^4,\qquad a_i\in\Z.\]
Given any integral binary quartic form $f$ and $(x_0,y_0)\in\Z^2$, define the action of \[\gamma=\begin{pmatrix}
a&b\\
c&d
\end{pmatrix}\in\GL_2(\Z)\] on the pair $(f,(x_0,y_0))$ by
\[\gamma\cdot (f(X,Y),(x_0,y_0))=(f((X,Y)\cdot \gamma),(x_0,y_0)\cdot\gamma^{-1} ),\]
where 
\[(X,Y)\cdot \gamma=(aX+cY,bX+dY).\]
This action preserves the value of $f(x_0,y_0)$.

We recall some facts about the seminvariants of quartic forms~\cite[Section~4.1.1]{Cremona}. For our convenience, we choose to scale the seminvariants differently than in~\cite{Cremona} , since we will only be dealing with integer-matrix binary forms.
The invariants of $f$ are
\begin{align*}
I=I(f)&=a_0a_4-4a_1a_3+3a_2^2,\text{ and }\\
J=J(f)&=a_0a_2a_4-a_0a_3^2-a_1^2a_4+2a_1a_2a_3-a_2^3.
\end{align*}
The discriminant of $f$ is 
\begin{align*}
\Delta (f)&\coloneqq I^3-27J^2\\
 & = a_0^3 a_4^3- 64 a_1^3a_3^3- 18 a_0^2 a_2^2 a_4^2 - 12 a_0^2a_1 a_3 a_4^2 - 6 a_0 a_1^2 a_3^2 a_4\\&\qquad
 - 180 a_0 a_1 a_2^2 a_3 a_4 
 + 81 a_0 a_2^4 a_4
 + 36 a_1^2a_2^2 a_3^2
 - 27 (a_0^2 a_3^4+a_1^4a_4^2) \\&\qquad
 + 54 a_2 (-a_2^2 + 2 a_1 a_3 + a_0 a_4) (a_4 a_1^2 + a_0 a_3^2)
\end{align*}
The seminvariants 
 attached to the form are $I$, $J$, $a=a(f)=a_0$,
\[H=H(f)=a_1^2-a_0a_2,\text{ and } R=R(f)=2a_1^3+a_0^2a_3-3a_0a_1a_2.\]
Comparing to the formulas in~\cite[Section~4.1.1]{Cremona}, here we have taken out a factor of $-48$ from their $H$, a factor of $32$ from their $R$, a factor of $12$ from their $I$, a factor $432$ from their $J$, and a factor of $256\cdot 27$ from their $\Delta$.
%The seminvariants $I,J,a,H,R$ together determine the form up to translation $X\mapsto X+kY$, $k\in\Z$.
The seminvariants are related by the syzygy
\begin{equation}\label{eq:syzygy}
H^3-\frac{I}{4}a^2H-\frac{J}{4}a^3=\left(\frac{R}{2}\right)^2.
\end{equation}
Notice that $(H,\frac{1}{2}R)$ defines an integral point on a twist of the elliptic curve $y^2=x^3-\frac{I}{4}-\frac{J}{4}$.

\subsection{Mordell's correspondence}
For integers $A,B\in\Z$ such that $4A^3+27B^2\neq 0$,
define an elliptic curve over $\Q$ with affine integral Weierstrass model
\[E_{A,B}:y^2=x^3+Ax+B.\]
The discriminant of $E_{A,B}$ is given by
\[\Delta_{E_{A,B}}=-16(4A^3+27B^2).\]
For integers $c,d,z\in\Z$,
define an integer-matrix binary quartic form
\[f_{c,d,e}(X,Y)=X^4+6cX^2Y^2+8d+eY^4.\]

Define
\begin{align*}
 & \mathcal{A}\coloneqq\{(E_{A,B},(x_0,y_0)): A,B\in\Z,\ 4A^3+27B^2\neq 0,\ (x_0,y_0)\in E_{A,B}(\Z)\},\\
&\mathcal{B}\coloneqq\{f_{c,d,e}:c,d, e\in \Z,\ e\equiv c^2\bmod 4,\ \Delta(f)\neq 0\}.
\end{align*}
The following correspondence is given by Mordell~\cite[Chapter 25]{Mordell} (or see~\cite[Section~2.3]{AlpogeHo} for a modern interpretation).
\begin{theorem}[Mordell]\label{thm:mordell}
Fix an integer $L\neq 0$.
There is a bijection
\[\mathcal{A}\rightarrow\mathcal{B}\]
given by 
\[(E_{A,B},(x_0,y_0))\mapsto f,\]
where \[f(X,Y)=X^4-6x_0X^2Y^2+8y_0XY^3+(-4A-3x_0^2)Y^4.\]
Moreover, under this map, $\Delta(f)=\Delta_{E_{A,B}}$, $I(f)=-4A$ and $J(f)= -4B$.
\end{theorem}
The inverse map comes from the syzygy~\eqref{eq:syzygy} satisfied by the seminvariants, but we will only make use of the the direction from $\mathcal{A}$ to $\mathcal{B}$ in Theorem~\ref{thm:mordell}.

\section{Lowering the discriminant}\label{sec:reduce}
We now fix an elliptic curve 
$E:y^2=x^3+Ax+B,\ A,B\in\Z$
and consider its quadratic twists
$E_D:y^2=x^3+AD^2x+BD^3$, where $D\in\D$.
For each $P=(c,d)\in E_D(\Z)$, Theorem~\ref{thm:mordell} gives the binary quartic form 
\begin{equation}\label{eq:fpdef}
f_P(X,Y)\coloneqq X^4-6cX^2Y^2+8dXY^3+(-4AD^2-3c^2)Y^4,
\end{equation}
 which satisfies $\Delta(f_P)=\Delta_ED^6$, $I(f_P) = -4AD^2$ and $J(f_P) = -4BD^3$.

Denote the space of integer-matrix binary quartic forms by $V$. 
Define
\begin{multline}\label{eq:Psi}
\Psi:\bigcup_{D\in\D}\left\{(P,M,k)\in E_D(\Z)\times \Z_{>0}\times (\Z/M\Z)^{\times}: 
\begin{aligned}
& M\mid D,\ \gcd(6\cdot x(P),M)=1,\\
&k^2\equiv x(P)\bmod M
\end{aligned}
\right\}\\
\rightarrow 
(V\times \Z^2)/\GL_2(\Z)
\end{multline}
given by 
\[
(P,M,k)\mapsto (F,(1,0)),
\]
where \[F(X,Y)=\frac{1}{M^3}
f_P(MX+kY,Y).\]
We will show that $\Psi$ is well-defined in Lemma~\ref{lemma:reddisc} and injective in Lemma~\ref{lemma:injective}

In work of Bombieri and Schmidt~\cite{BomSch}, to bound the number of solutions to a Thue equation $F_1(X,Y)=h$, they transformed the integral binary form $F_1$ to a different form $F_2$, whose discriminant is raised by a factor of $h^6$, and so that there is a solution to $F_2(X,Y)=1$. Some applications of this idea can be found in~\cite{AB,akhtarilocalglobal}. Here we attempt to carry out the reverse process on the integral quartic forms $f_P$ to lower their discriminants.

\begin{lemma}\label{lemma:reddisc}
Let $P=(c,d)\in E_D(\Z)$ and take $f_P$ as defined in~\eqref{eq:fpdef}. 
Fix a positive squarefree integer $M$ dividing $D$ that is coprime to $6c$.
Then $c$ is a square modulo $M$, and for any integer $k$ such that $k^2=c\bmod M$, we have that
\[
F(X,Y)\coloneqq 
% \frac{1}{M} 
% \left(\begin{pmatrix}
% M & 0\\
% k & 1
% \end{pmatrix}\cdot
% f_P\right)(X,Y)=
\frac{1}{M^3}\cdot
f_P\left((X,Y)\cdot \begin{pmatrix}
M & 0\\
k & 1
\end{pmatrix}\right)=
\frac{1}{M^3}\cdot
f_P(MX+kY,Y)
\]
is an integer-matrix binary quartic form. Moreover, we have
\begin{enumerate}
 \item $F(1,0)=M$, 
 \item $I(F) = -4A(D/M)^2$, $J(F) =-4B(D/M)^3$, and 
 \item$\Delta(F)=\Delta(f_P)/M^6$.
\end{enumerate}
\end{lemma}

\begin{proof}
Since $M$ is squarefree, $M$ is a product of distinct prime factors.
Since $(c,d)\in\E_D(\Z)$, we have $d^2=c^3-D^2c$. For each prime $p\mid M$, we see that $c^3\equiv d^2\bmod p$, so $(\frac{c}{p})=1$, therefore $c$ is a square modulo $M$.
Taking any integer $k$ such that $k^2\equiv c\bmod M$, by Hensel's lemma we can find a lift of $k$ such that $k\equiv K\bmod M$ and
\begin{equation}\label{eq:cmod}
c\equiv K^2\bmod M^3.
\end{equation}
It suffices to show that $F$ is an integer-matrix binary quartic form with this choice of $K$, since otherwise $k=K+uM$ for some integer $u$, and we can consider $F(X-uY,Y)$ instead.

Putting~\eqref{eq:cmod} into $d^2=c^3+AD^2c+BD^3$, we see that
\begin{equation}\label{eq:dmod}
d\equiv K^3+\frac{AD^2}{2K}\bmod M^3.
\end{equation}
By~\eqref{eq:cmod} and~\eqref{eq:dmod}, we see that the coefficients of 
\begin{multline*}
f_P(MX+kY,Y)
=M^4X^4+4M^3KX^3y+6M^2(K^2-c)X^2Y^2\\
+4M(K^3-3cK+2d)XY^3+(K^4-6cK^2+8dK-4AD^2-3c^2)Y^4
\end{multline*}
are all divisible by $M^3$.
Therefore $F$ is integer-matrix from the coefficients. The remaining properties are then a straightforward check from the definition of $F$.
\end{proof}

\begin{lemma}\label{lemma:injective}
The map $\Psi$ is injective.
\end{lemma}
\begin{proof}
The value of $F(1,0)$ determines $M$, and together with the discriminant of $F$, determines $D$.
In the following, fix some $D\in\D$ and some $M\mid D$ such that $\gcd(6,M)=1$.
Suppose $P,Q\in E_D(\Z)$ satisfy $\gcd(x(P),M)=\gcd(x(Q),M)=1$ and write $\Psi(P,M,k_P)=(F_P,(1,0))$ and $\Psi(P,M,k_Q)=(F_Q,(1,0))$.
Suppose $(F_P,(1,0))$ and $(F_Q,(1,0))$ are $\GL_2(\Z)$-equivalent, so $\gamma\cdot(F_P,(1,0))=(F_Q,(1,0))$ for some $\gamma \in\GL_2(\Z)$. Then $(1,0)\cdot \gamma^{-1}=(1,0)$ implies that we can write $\gamma=\begin{pmatrix}
1 & 0\\
u & 1\end{pmatrix}$ for some $u\in\Z$.
Recall that \[F_P(X,Y)=\frac{1}{M^3}\cdot
f_P\left((X,Y)\cdot \begin{pmatrix}
M & 0\\
k_P & 1
\end{pmatrix}\right)\]
and 
\[F_Q(X,Y)=\frac{1}{M^3}\cdot
f_Q\left((X,Y)\cdot \begin{pmatrix}
M & 0\\
k_Q & 1
\end{pmatrix}\right)\]
From $F_P((X,Y)\cdot \gamma)=F_Q(X,Y)$, we get
 \[
f_P\left((X,Y)\cdot \gamma\cdot \begin{pmatrix}
M & 0\\
k_P & 1
\end{pmatrix}\right)
=
f_Q\left((X,Y)\cdot \begin{pmatrix}
M & 0\\
k_Q & 1
\end{pmatrix}\right).\]
Then since
\[\begin{pmatrix}
M & 0\\
k_Q & 1\end{pmatrix}^{-1}\begin{pmatrix}
1 & 0\\
u & 1\end{pmatrix}\begin{pmatrix}
M & 0\\
k_P & 1\end{pmatrix}
=\begin{pmatrix}
1 & 0\\
uM+k_P-k_Q & 1\end{pmatrix},
\]
we have
\[f_P\left((X,Y)\cdot \begin{pmatrix}
1 & 0\\
uM+k_P-k_Q & 1\end{pmatrix}\right)=
f_Q(X,Y).
\]
The $X^3Y$-coefficients of $f_P$ and $f_Q$ are both $0$, it must be that $uM+k_P-k_Q=0$ and $f_P=f_Q$. Hence $P=Q$ and $k_P\equiv k_Q\bmod M$.
\end{proof}

Heath-Brown~\cite{HBSelmer1,HBSelmer} computed the moments of the size of the $2$-Selmer group modulo $2$-torsion in this family is~$3$. We will extract some information about the $2$-Selmer elements in this family from the argument in~\cite{HBSelmer1,HBSelmer}, in order to show that we can usually pick a suitable $M$ to apply Lemma~\ref{lemma:reddisc}.

\section{The $2$-Selmer group of $y^2=x^3-D^2x$}\label{sec:Selmer}
In the following sections we will specialise in the case when $A=-1$ and $B=0$, that is, the quadratic twists family
\[\E_D:y^2=x^3-D^2x.\]

Heath-Brown~\cite{HBSelmer1,HBSelmer} computed the moments of the size of the $2$-Selmer group modulo $2$-torsion in this family is~$3$. We will extract some information about the $2$-Selmer elements in this family from the argument in~\cite{HBSelmer1,HBSelmer}, in order to show that we can usually pick a suitable $M$ to apply Lemma~\ref{lemma:reddisc}.

The $2$-Selmer group of $\E_D$ is defined to be
\[\Sel_2(\E_D)\coloneqq\ker\left(H^1(\Gal(\overline{\Q}/\Q),\E_D[2])\rightarrow \prod_{p\text{ place of }\Q}H^1(\Gal(\overline{\Q}_p/\Q_p),\E_D)\right).\]
Since $\E_D$ has full $2$-torsion, there is an isomorphism $H^1(\Gal(\overline{\Q}/\Q),\E_D[2])\cong ((\Q^\times/(\Q^\times)^2)^2$, and it is possible to obtain explicit equations for the homogeneous spaces (See for example~\cite[Chapter~X, Proposition 1.4]{Silverman}). 
Explicit equations for homogeneous spaces for the curves $\E_D$ were found as part of Heath-Brown's argument~\cite[Section~2]{HBSelmer1}. As we will see, each $2$-Selmer element of $\E_D(\Q)$ corresponds to a system of two binary quadratic forms that is everywhere locally solvable. We will follow~\cite[Section~2]{HBSelmer1} to recover the equations. 

\begin{definition}
For $D\in \D$, define $\mathcal{W}_D$ to be the set of all $(D_1,D_2,D_3,D_4)\in\D^4$ such that
\begin{enumerate}
\item the system
\begin{equation}\label{eq:homeq}
D_1X^2+D_4W^2=D_2Y^2,\ D_1X^2-D_4W^2=D_3Z^2,
\end{equation}
 is everywhere locally solvable, and
\item $D_1D_2D_3D_4=D$.
\end{enumerate}
\end{definition}
Consider the injective homomorphism
\begin{equation}\label{eq:theta}
\theta:\E_D(\Q)/2\E_D(\Q)\rightarrow \Q^\times/(\Q^\times)^2\times \Q^\times/(\Q^\times)^2\times \Q^\times/(\Q^\times)^2
\end{equation}
given by 
$(x,y)\mapsto (x-D,x,x+D)$
at non-torsion points.
At torsion points $\theta(O)=(1,1,1)$, $\theta((0,0))=(-D,-1,D)$, $\theta((D,0))=(2,D,2D)$, $\theta((-D,0))=(-2D,-D,1)$.
\begin{lemma}\label{lemma:identifyS2}
The set $\mathcal{W}_D$ is in bijection to
\[
\begin{cases}
\Sel_2(\E_D)/\psi(\{O,(0,0),(\pm D,0)\}) &\text{if }D\text{ is odd},\\
\hfil \Sel_2(\E_D)/\psi(\{O,(0,0)\}) &\text{if }D\text{ is even},
\end{cases}
\]
where $\psi$ here denotes the natural map
\[\psi:\E_D(\Q)\twoheadrightarrow\E_D(\Q)/2\E_D(\Q)\hookrightarrow\Sel_2(\E_D).\] 
In particular, if 
the image of $(c,d)\in\E_D(\Z)$ under the map 
\[\E_D(\Z)\hookrightarrow \E_D(\Q)\twoheadrightarrow\E_D(\Q)/2\E_D(\Q)\hookrightarrow\Sel_2(\E_D)\twoheadrightarrow\mathcal{W}_D\]
is $(D_1,D_2,D_3,D_4)$, then 
\begin{equation}\label{eq:gcdcD}
 \gcd(c,D)\in\{D_1D_2D_3,\ D_2D_3D_4,\ D_1D_2D_4, \ D_1D_3D_4\}.
\end{equation}
\end{lemma}
\begin{proof}
We study $\im\theta$ following~\cite[Section~2]{HBSelmer1}.
Suppose $(x,y)\in \E_D(\Q)$, and write $x=r/s$ and $y=t/u$, where $r,s,t,u$ are integers and $\gcd(r,s)=1$ and $\gcd(t,u)$=1.
Putting this into $y^2=x^3-D^2x$, 
we have 
\[r(r+sD)(r-sD)u^2=t^2s^3.\]
Then since $\gcd(t,u)=1$ and $\gcd(r,s)=1$, we must have $s^3=u^2$, so $s=W^2$ for some integer $W$.
Now write $\gcd(r,D)=D_0$, and $r=D_0r'$. From
\[r(r+sD)(r-sD)=t^2\]
we see that $D_0^3\mid t^2$, hence $D_0^2\mid t$ since $D_0$ is squarefree.
Then writing $D_4=D/D_0$, we have $\gcd(r',sD_4)=1$ by construction, and we have
\[r'(r'+sD_4)(r'-sD_4)=D_0(t/D_0^2)^2.\]
The factors on the left are pairwise coprime except possibly a common factor of $2$, which only occurs when $r'$ and $sD_4$ are both odd, in this case $r',(r'+sD_4)/2,(r'-sD_4)/2$ are pairwise coprime.
Now we can write $r'=D_1X^2$, $r'+sD_4=D_2Y^2$ and $r'-sD_4=D_3Z^2$, where $D_1,D_2,D_3$ are squarefree integers such that $D_1D_2D_3=D_0$ or $D_1D_2D_3=4D_0$. In the first case $D_1D_2D_3D_4=D$ and in the second case $D_1D_2D_3D_4=4D$ with $D_1,D_4$ odd and $D_2,D_3$ even.
When $D$ is even, the case $D_1D_2D_3D_4=4D$ is not possible since $D_1,D_4$ would need to be odd, and $D_2,D_3$ would need to be even and squarefree.
This produces a solution to the system
\[
D_1X^2+D_4W^2=D_2Y^2,\ D_1X^2-D_4W^2=D_3Z^2.
\]
An element in $\Sel_2(\E_D)$ corresponds to a system of this form
that is everywhere locally solvable. To fix the signs of $D_1,D_2,D_3,D_4$ and their valuations at $2$, we will use the torsion points of $\E_D$.

If $D$ is odd, exactly one of $(x',y')\in\{(x,y), (x,y)+(0,0), (x,y)+(D,0), (x,y)+(-D,0)\}$ satisfies $x'>0$ and $v_2(x')\neq 0$. By looking at $\theta((x',y'))$, we see that the image of $(x,y)$ in $\mathcal{W}_D$ can be taken as
\[
\begin{cases}
\hfil (D_1,D_2,D_3,D_4) &\text{if }v_2(x)\neq 0\text{ and }x>0,\\
\hfil (D_4,D_3,-D_2,-D_1) &\text{if }v_2(x)\neq 0\text{ and }x<0,\\
\hfil (D_2/2,D_1,D_4,D_3/2) &\text{if }v_2(x)= 0\text{ and }x>0,\\
(D_3/2,D_4,-D_1,-D_2/2) &\text{if }v_2(x)= 0\text{ and }x<0.\\
\end{cases}
\]
If $D$ is even, exactly one of $(x',y')\in\{(x,y), (x,y)+(0,0)\}$ satisfies $x'>0$ and $v_2(x')\equiv 1\bmod 2$. We take the image of $(x,y)$ in $\mathcal{W}_D$ to be 
\[
\begin{cases}
\hfil(D_1,D_2,D_3,D_4) &\text{if }\text{ and }x>0,\\
(D_4,D_3,-D_2,-D_1) &\text{if }\text{ and }x<0.
\end{cases}
\]

If $(x,y)\in\E_D(\Z)$, by construction $\gcd(x,D)=D_0$, which is $D_1D_2D_3$ or $\frac{1}{4}D_1D_2D_3$. This gives~\eqref{eq:gcdcD} by checking each of the above cases and using the different labelling of $D_1,D_2,D_3,D_4$ in the statement of the lemma.
\end{proof}

% We see that under this construction if $(x,y)\in \E_D(\Q)$ is sent to $(D_1,D_2,D_3,D_4)\in\mathcal{W}_D$, then
% \begin{equation}\label{eq:homw}
% \theta((x,y))\in
% \begin{cases}
% \left\{\begin{gathered}
% (D_1D_2,D_2D_3,D_1D_3),\ (-D_3D_4,-D_2D_3,D_2D_4),\\
% (2D_1D_2,D_1D_4,2D_2D_4),\ (-2D_3D_4,-D_1D_4,D_1D_3)
% \end{gathered}\right\} &\text{if }D\text{ is odd},\\
% \hfil \{(D_1D_2,D_2D_3,D_1D_3),\ (-D_3D_4,-D_2D_3,D_2D_4)\} &\text{if }D\text{ is even}.
% \end{cases}
% \end{equation} 

\section{Generic $2$-Selmer elements}\label{sec:genericSelmer}
We want to show that those $(D_1,D_2,D_3,D_4)$ that appears usually satisfies some nice properties that will eventually allow us to pick $M$ for Lemma~\ref{lemma:reddisc}.
Take 
\[N^{\ddagger}=\exp(4\kappa(\log\log N)^2),\] where $\kappa>0$ is an absolute constant as defined in~\cite[Lemma~7]{HBSelmer1} or~\cite[Section~3]{HBSelmer}. This $N^{\ddagger}$ is $C^4$ in the notation of~\cite{HBSelmer1}.

Henceforth $0<\epsilon<\frac{1}{2}$ will be a fixed constant. 
Let $S$ be the interval
\[S\coloneqq\left[\exp((\log N)^{2\epsilon}),\ \exp((\log N)^{1-2\epsilon})\right].\]
so that any $p\in S$ satisfies 
\[2\epsilon\log\log N<\log\log p<(1-2\epsilon)\log\log N.\]
Define
\begin{align*}
\omega(n)&\coloneqq\#\{p\text{ prime}:p\mid n\},\\
 \omega_S(n)&\coloneqq\#\{p\text{ prime}:p\mid n,\ p\in S\}.
\end{align*}

We will prove the following.
\begin{theorem}\label{thm:cond}
Define two properties on $(D_1,D_2,D_3,D_4)\in\mathcal{W}_D$:
\begin{enumerate}[label=(\textbf{S\arabic*})]
\item \label{cond1}for each $i=1,2,3,4$, we have $D_i\geq N^{\ddagger}
$ and there exist some $p\mid D_i$ such that $p\in S$;
\item \label{cond2}$(D_1,D_2,D_3,D_4)$ corresponds to a torsion point on $\E_D(\Q)$.
\end{enumerate}
Then
\[
 \#\left\{D\in\D_N:\ref{cond1}\text{ or }\ref{cond2}\text{ fails for some }
 (D_1,D_2,D_3,D_4)\in\mathcal{W}_D\right\}\ll N(\log N)^{-\frac{1}{4}+\epsilon}.
\]
\end{theorem}

Define \[\delta_i(\eta)\coloneqq\begin{cases}
1 &\text{if } \eta=i,\\
0 &\text{otherwise}.
\end{cases}
\]
To prove Theorem~\ref{thm:cond}, it suffices to bound the number of $4$-tuples of positive odd integers $(D_1,D_2,D_3,D_4)$ satisfying the following conditions:
\begin{enumerate}
\item $(2^{\delta_1(\eta)}D_1,2^{\delta_2(\eta)}D_2,2^{\delta_3(\eta)}D_3,2^{\delta_4(\eta)}D_4)\in\mathcal{W}_D$ for some $D\in\D_N$ and some $\eta\in\{0,1,2,3,4\}$
, and 
\item one of the conditions~\ref{cond1w} and~\ref{cond2w} listed below.
\end{enumerate}
\begin{enumerate}[label=(\textbf{W\arabic*})]
 \item \label{cond1w} For some $i=1,2,3,4$, we have $D_{i}<N^{\ddagger}$, and \[(D_1,D_2,D_3,D_4)\neq 
 \begin{cases}
 \hfil (1,1,1,D) &\text{if }\eta=0\text{ or }4,\\
 (1,1,D/2,1)&\text{if }\eta=2.
 \end{cases}\]
\item \label{cond2w} We have $D_{i}
\geq N^{\ddagger}$ for all $i=1,2,3,4$, and there exists an $i$ such that $D_i$ has no prime factor in $S$.
\end{enumerate}
In the above notation, $\eta=0$ implies that $D=D_1D_2D_3D_4$ is odd, and $\eta=1,2,3,4$ implies that $D=2D_1D_2D_3D_4$ is even.

We first consider the case when $D$ is odd.
By~\cite[Lemma~3]{HBSelmer1}, considering the local solvability of~\eqref{eq:homeq} at each prime, we can package the condition $(D_1,D_2,D_3,D_4)\in \mathcal{W}_D$ as a sum of product of Jacobi symbols.
Write $\mathbf{D}=(D_{ij})$ as a $16$-tuple of positive odd integers, where the indices are in the range
\[1\leq i\leq 4,\ 0\leq j\leq 4,\ i\neq j.\]
For odd $D$, set
\[g_0(\mathbf{D})=\leg{-1}{\alpha}\leg{2}{\beta_0}\prod_i 4^{-\omega(D_{i0})}\prod_{j\neq 0}4^{-\omega(D_{ij})}\prod_{k\neq i,j}\prod_{l}\leg{D_{kl}}{D_{ij}},\]
where
$\alpha = D_{12}D_{14}D_{23}D_{21}$ and $\beta_0= D_{24}D_{21}D_{34}D_{31}$.
Then
\[G_0(D_1,D_2,D_3,D_4)\coloneqq \sum_{\mathbf{D}: \prod_{j\neq i} D_{ij}=D_i}g_{\eta}(\mathbf{D})
=\begin{cases}
1&\text{if }(D_1,D_2,D_3,D_4)\in \mathcal{W}_D\\
0&\text{else.}
\end{cases}
\]

For even $D$, since we are only aiming for an upper bound, we can ignore any local conditions at $2$ when considering the solvability of~\eqref{eq:homeq}. Following the proof of~\cite[Lemma~3]{HBSelmer1}, the only difference being essentially pulling the factor of $2$ in $D$ into the $\leg{2}{\cdot}$ term, set
\[g_{\eta}(\mathbf{D})=\leg{-1}{\alpha}\leg{2}{\beta_{\eta}}
\prod_i 4^{-\omega(D_{i0})}\prod_{j\neq 0}4^{-\omega(D_{ij})}\prod_{k\neq i,j}\prod_{l}\leg{D_{kl}}{D_{ij}},\]
where 
\begin{align*}
&\beta_1=D_{23}D_{32}D_{42}D_{43}D_{21}D_{31}, \qquad 
\beta_2=D_{13}D_{14}D_{41}D_{43}D_{24}D_{21}, \\
&\beta_3=D_{12}D_{14}D_{41}D_{42}D_{34}D_{31}, \qquad 
\beta_4=D_{12}D_{13}D_{23}D_{32}D_{24}D_{34}.
\end{align*}
Then 
\[\begin{split}
G_{\eta}(D_1,D_2,D_3,D_4)
&\coloneqq \sum_{\mathbf{D}: \prod_{j\neq i} D_{ij}=D_i}g_{\eta}(\mathbf{D})\\
&\geq \begin{cases}
1&\text{if }(2^{\delta_1(\eta)}D_1,2^{\delta_2(\eta)}D_2,2^{\delta_3(\eta)}D_3,2^{\delta_4(\eta)}D_4)\in \mathcal{W}_{D}\\
0&\text{else.}
\end{cases}
\end{split}\]

\subsection{The case $D_i<N^{\ddagger}$}
For each $\eta=0,1,2,3,4$, we will want to estimate the sum 
\[
\sum_{\substack{(D_1,D_2,D_3,D_4)\\~\ref{cond1w}}}G_{\eta}(D_1,D_2,D_3,D_4)\]
where the sum is taken over all positive odd integers $D_1,D_2,D_3,D_4$ such that $D_1D_2D_3D_4\in \D_N$.
 Following~\cite[Section~3]{HBSelmer1} and~\cite[Section~3]{HBSelmer}, dissect the sum according to the size of each $D_{ij}$ in the factorisation.
For each $(i,j)$, take $A_{ij}$ to run over powers of $2$.
For $\mathbf{A}=(A_{ij})$, define the restricted sum
\[|S_{\eta}(\mathbf{A})|=\sum_{\substack{\mathbf{D}\\A_{ij}< D_{ij}\leq 2 A_{ij}}}g_{\eta}(\mathbf{D}),\]
where the sum is taken over all 16-tuples of odd positive integers $\mathbf{D}=(D_{ij})$ such that $\prod_{i,j} D_{ij}\in \D_N$ and $A_{ij}< D_{ij}\leq 2 A_{ij}$ for every $i,j$.
Note that if $A_{ij}=\frac{1}{2}$, the interval $A_{ij}< D_{ij}\leq 2 A_{ij}$ forces $D_{ij}=1$.

We can summarise the error term estimates in~\cite{HBSelmer1,HBSelmer} as follows.
\begin{lemma}[{\cite[Lemma~7, Lemma~11]{HBSelmer1},~\cite[Lemma~6]{HBSelmer}}]\label{lemma:HBsum}
\[\sum_{\mathbf{A}}|S_{\eta}(\mathbf{A})|\ll N(\log N)^{\frac{1}{4}+\epsilon},\]
where the sum is over all $\mathbf{A}$ other than those that satisfy
\begin{equation}\label{eq:Amaxunlinked}
\begin{cases}
 A_{\mathbf{u}}>N^{\ddagger}&\text{if }\mathbf{u}\in \mathcal{U}\text{, and}\\
 \hfil A_{\mathbf{u}}=\frac{1}{2}&\text{if }\mathbf{u}\notin \mathcal{U},
\end{cases}
\end{equation}
 for all indices $\mathbf{u}$, where $\mathcal{U}$ is one of
\begin{equation}\label{eq:Usets}
 \begin{cases}
 \{10, 20, 30, 40\}, \{ 40, 41, 42, 43\}, \{20, 12, 32, 42\}, \{30, 13, 23, 43\}&\text{if }\eta=0,\\
\{10,20,30,40\},\{40,14,24,34\} &\text{if }\eta=1,\\
\{10,20,30,40\},\{20,12,22,32\},\{30,31,32,34\} &\text{if }\eta=2,\\
\{10,20,30,40\},\{30,13,23,33\} &\text{if }\eta=3,\\
\{10,20,30,40\},\{10,11,21,31\},\{40,41,42,43\} &\text{if }\eta=4.
\end{cases}
\end{equation}
Moreover, the estimate still holds if we further impose that $\prod_j D_{ij}\leq N_i$ for each $i=1,2,3,4$.
\end{lemma}
Strictly speaking~\cite{HBSelmer1} only applies to the case for odd $D$, namely $\eta=0$.
The modification made in Lemma~\ref{lemma:HBsum} for even $D$ was in the sets in~\eqref{eq:Usets}, since the possible $\mathcal{U}$ depends on the variables in $\beta_{\eta}$.
Following the case analysis in the proof of~\cite[Lemma~11]{HBSelmer1} with this change, the sets in~\eqref{eq:Usets} can be found for each of $\eta=1,2,3,4$.

We are ready to bound the contribution from~\ref{cond1w}.
\begin{lemma}\label{lemma:w1} For each $\eta\in\{0,1,2,3,4\}$, we have
\[ \#\{ (2^{\delta_1(\eta)}D_1,2^{\delta_2(\eta)}D_2,2^{\delta_3(\eta)}D_3,2^{\delta_4(\eta)}D_4)\in\mathcal{W}_D:\ref{cond1w}\text{ holds}
 \} \ll N(\log N)^{-\frac{1}{4}+\epsilon}.\]
\end{lemma}
\begin{proof}
We want to bound
\begin{equation}\label{eq:oddG}
\sum_{\substack{(D_1,D_2,D_3,D_4)\\~\ref{cond1w}}}G_{\eta}(D_1,D_2,D_3,D_4)=\sum_{\substack{(D_1,D_2,D_3,D_4)\\~\ref{cond1w}}}\sum_{\substack{\mathbf{D}\\ \prod_{j\neq i} D_{ij}=D_i}}g_{\eta}(\mathbf{D}).
\end{equation}
Dissect the sum using $\mathbf{A}$, then apply Lemma~\ref{lemma:HBsum} to~\eqref{eq:oddG}.
The only possibility that this $S_{\eta}(\mathbf{A})$ is not covered by Lemma~\ref{lemma:HBsum} is if $\mathbf{A}$ satisfies~\eqref{eq:Amaxunlinked} with \[\mathcal{U}=\begin{cases}
\{40,41,42,43\}&\text{when }\eta=0,\\
\{30,31,32,34\}&\text{when }\eta=2,\\
\{40,41,42,43\}&\text{when }\eta=4.
\end{cases}
\] 
This implies that 
\[(D_1,D_2,D_3,D_4)=
 \begin{cases}
 \hfil (1,1,1,D) &\text{if }\eta=0\text{ or }4,\\
 (1,1,D/2,1)&\text{if }\eta=2,
 \end{cases}\] which are from the excluded torsion points.
Therefore the required estimate follows from that in Lemma~\ref{lemma:HBsum}.
\end{proof}

\subsection{Prime divisor of a large $D_i$}
We now bound the contribution from~\ref{cond2w}.
Assume that it is $D_4\geq N^{\ddagger}$ that is not divisible by any prime in $S$. The cases with $D_4$ replaced by $D_1$, $D_2$, $D_3$ the same after relabelling.
We want to bound
\begin{equation}\label{eq:expectedprime}
\sum_{\substack{D_1D_2D_3D_4\in\D_N\\D_1,D_2,D_3,D_4\geq N^{\ddagger}\\ p\mid D_4 \Rightarrow p\notin S}}G_{\eta}(D_1,D_2,D_3,D_4)
=\sum_{\substack{D_1D_2D_3D_4\in\D_N\\D_1,D_2,D_3,D_4\geq N^{\ddagger}\\ p\mid D_4 \Rightarrow p\notin S}}\sum_{\substack{\mathbf{D}\\ \prod_{j\neq i} D_{ij}=D_i}}g_{\eta}(\mathbf{D}).
\end{equation}
For each $\mathbf{D}$ that appears in the sum of $g_{\eta}(\mathbf{D})$, we can find a set of indices $\mathcal{U}=\{1i,2j,3k,4l\}$, where $i,j,k,l$ are not necessarily distinct, such that $D_{1i}$, $D_{2j}$, $D_{3k}$, $D_{4l}>(N^{\ddagger})^{\frac{1}{4}}$. 

Lemma~\ref{lemma:HBsum} do not immediately apply because of the restrictions on the primes dividing $D_4$. However it will be sufficient to follow~\cite[Section~3]{HBSelmer1}.

\begin{definition}
We call two indices $(i,j)$ and $(k,l)$ linked if 
\[i\neq k\text{ and
precisely one of the conditions }
\begin{cases}
l\neq 0,i,\\
j\neq 0, k
\end{cases} \text{ holds.}\]
\end{definition}
We will show that the indices in $\mathcal{U}$ are pairwise unlinked with the number of exceptions contributing $O(N(\log N)^{-\frac{1}{4}+\epsilon})$ to~\eqref{eq:expectedprime}.

When there are linked indices $\mathbf{u}$ and $\mathbf{v}$, the Jacobi symbol between $D_{\mathbf{u}}$ and $D_{\mathbf{v}}$ appears in $g_{\eta}$ non-trivially. When both $D_{\mathbf{u}}$ and $D_{\mathbf{v}}$ are large, then the large sieve (\cite[Lemma~4]{HBSelmer1}) can be applied to obtain cancellation. Imposing further restrictions to each of $D_{\mathbf{u}}$, $D_{\mathbf{v}}$, $(D_{\mathbf{w}})_{\mathbf{w}\neq \mathbf{u},\mathbf{v}}$ do not affect the estimates in~\cite[Section~3]{HBSelmer1}. We state the modified statement below.

\begin{lemma}[{\cite[Lemma~5]{HBSelmer1}}]\label{lemma:largesieve}
Suppose $\mathbf{u}$ and $\mathbf{v}$ are linked indices.
Take $\mathcal{C}_1$ to be any collection of $D_{\mathbf{u}}\geq (\log N)^{544}$.
Take $\mathcal{C}_2$ to be any collection of $D_{\mathbf{v}}\geq (\log N)^{544}$.
Take $\mathcal{C}_3$ to be any collection of $\mathbf{B}=(D_{\mathbf{w}})_{\mathbf{w}\neq \mathbf{u},\mathbf{v}}$.
Then
\[\sum_{D_{\mathbf{u}}\in\mathcal{C}_1}
\sum_{D_{\mathbf{v}}\in\mathcal{C}_2}
\sum_{\substack{\mathbf{B}\in\mathcal{C}\\ \prod D_{ij}\in\D_N}}
g_{\eta}(\mathbf{D})\ll N(\log N)^{-17}.\]
\end{lemma}

Suppose $\mathbf{u}$ and $\mathbf{v}$ are linked.
When one of $D_{\mathbf{u}}$ is large, and $D_{\mathbf{v}}\neq 1$ is small, Siegel--Walfisz for character sums (\cite[Lemma~6]{HBSelmer1}) is used instead. For this we need that the sum over $D_{\mathbf{u}}$ is over an interval, but we can still impose restrictions on $D_{\mathbf{v}}$ and $(D_{\mathbf{w}})_{\mathbf{w}\neq \mathbf{u},\mathbf{v}}$ without affecting the estimates.

\begin{lemma}[{\cite[Lemma~7]{HBSelmer1}}]\label{lemma:SW}
Suppose $\mathbf{u}$ and $\mathbf{v}$ are linked indices.
Take $\mathcal{C}_2$ to be any collection of $1<D_{\mathbf{v}}< (\log N)^{544}$.
Take $\mathcal{C}_3$ to be any collection of $\mathbf{B}=(D_{\mathbf{w}})_{\mathbf{w}\neq \mathbf{u},\mathbf{v}}$.
Define
\[S_{1,\eta}(A_{\mathbf{u}})\coloneqq
\sum_{A_{\mathbf{u}}<D_{\mathbf{u}}\leq 2A_{\mathbf{u}}}
\sum_{D_{\mathbf{v}}\in\mathcal{C}_2}
\sum_{\substack{\mathbf{B}\in\mathcal{C}_3\\ \prod D_{ij}\in\D_N}}
g_{\eta}(\mathbf{D}).\]
Then
\[\sum_{A_{\mathbf{u}}}|S_{1,\eta}(A_{\mathbf{u}})|\ll N(\log N)^{-17},\]
where the sum is over all $A_{\mathbf{u}}\geq (N^{\ddagger})^{\frac{1}{4}}$ that are powers $2$.
\end{lemma}

Combining Lemma~\ref{lemma:largesieve} and Lemma~\ref{lemma:SW} we have the following.
\begin{lemma}\label{lemma:Scombined}
Suppose $\mathbf{u}$ and $\mathbf{v}$ are linked indices.
Take $\mathcal{C}_2$ to be any collection of $D_{\mathbf{v}}>1$.
Take $\mathcal{C}_3$ to be any collection of $\mathbf{B}=(D_{\mathbf{w}})_{\mathbf{w}\neq \mathbf{u},\mathbf{v}}$.
Define
\[S_{1,\eta}(A_{\mathbf{u}})\coloneqq
\sum_{A_{\mathbf{u}}<D_{\mathbf{u}}\leq 2A_{\mathbf{u}}}
\sum_{D_{\mathbf{v}}\in\mathcal{C}_2}
\sum_{\substack{\mathbf{B}\in\mathcal{C}_3\\ \prod D_{ij}\in\D_N}}
g_{\eta}(\mathbf{D}).\]
Then
\[\sum_{A_{\mathbf{u}}}|S_{1,\eta}(A_{\mathbf{u}})|\ll N(\log N)^{-17},\]
where the sum is over all $A_{\mathbf{u}}\geq (N^{\ddagger})^{\frac{1}{4}}$ that are powers $2$.
\end{lemma}

We can now return to the set $\mathcal{U}=\{1i,2j,3k,4l\}$. 
Since by construction $D_{\mathbf{u}}>(N^{\ddagger})^{\frac{1}{4}}$ for all $\mathbf{u}\in\mathcal{U}$, we can assume that the indices $1i,2j,3k,4l$ are pairwise unlinked by Lemma~\ref{lemma:largesieve}.

Now suppose $\mathbf{v}\notin\mathcal{U}$. If $\mathbf{v}$ is not linked any one of $1i,2j,3k$, then $\{1i,2j,3k,\mathbf{v}\}$ is a set of unlinked indices.
Comparing against the $24$ possible sets of unlinked indices in~\cite[Lemma~9]{HBSelmer1}, if $\{1i,2j,3k,4l\}$ and $\{1i,2j,3k,\mathbf{v}\}$ are both sets of unlinked indices, they must be the same set.
Therefore $\mathbf{v}$ must be linked to one of $\{1i,2j,3k\}$, and this allows us to apply Lemma~\ref{lemma:Scombined}, so we are left with the terms in the sum~\eqref{eq:expectedprime} with $D_{\mathbf{v}}=1$ for all $\mathbf{v}\notin\mathcal{U}$.
Noting that there are only a maximum of 24 possible sets of unlinked indices $\mathcal{U}$, and putting in $D_1=D_{1i}$, $D_2=D_{2j}$, $D_3=D_{3k}$, $D_4=D_{4l}$, the sum~\eqref{eq:expectedprime} is bounded by 
\begin{align*}
&\ll\sum_{\substack{D_1D_2D_3D_4\in\D_N\\D_1,D_2,D_3,D_4\geq N^{\ddagger}\\ p\mid D_4 \Rightarrow p\notin S}}4^{-\omega(D_1D_2D_3D_4)}+O\left(N(\log N)^{-\frac{1}{4}+\epsilon}\right)
\\
&\leq\sum_{D\in\D_N}4^{-\omega(D)}\sum_{\substack{(D_1,D_2,D_3,D_4)\\D_1D_2D_3D_4=D\\ p\mid D_4 \Rightarrow p\notin S}}1
+O\left(N(\log N)^{-\frac{1}{4}+\epsilon}\right)\\
&
=\sum_{D\in\D_N} \left(\frac{3}{4}\right)^{\omega_S(D)}+O\left(N(\log N)^{-\frac{1}{4}+\epsilon}\right).
\end{align*}
To bound the main term here we make use of the following result.
\begin{lemma}[{\cite[Theorem~1]{Shiu}}]\label{lemma:Shiu}
Fix $0<\epsilon<1$ and some positive constant $C$.
Let $f$ be a multiplicative function such that $f(p^{\ell})\leq C $ for all prime $p$ and $\ell\geq 1$.
Then
\[\sum_{x-y<n\leq x}f(n)\ll
\frac{y}{\log x}\exp\left(\sum_{p\leq x}\frac{f(p)}{p}\right)
\]
uniformly for
$2\leq X^{1-\epsilon}\leq Y< X$.
\end{lemma}
By Lemma~\ref{lemma:Shiu} and Mertens' theorem, 
\[
\sum_{\substack{D_1D_2D_3D_4\in\D_N\\D_1,D_2,D_3,D_4\geq N^{\ddagger}\\ p\mid D_4 \Rightarrow p\notin S}}G_{\eta}(D_1,D_2,D_3,D_4)\ll\sum_{D\in\D_N} \left(\frac{3}{4}\right)^{\omega_S(D)}\ll N(\log N)^{-\frac{1}{4}+\epsilon}
\]
This gives the following estimate.
\begin{lemma}\label{lemma:w2}For each $\eta\in\{0,1,2,3,4\}$, we have
\[ \#\{(2^{\delta_1(\eta)}D_1,2^{\delta_2(\eta)}D_2,2^{\delta_3(\eta)}D_3,2^{\delta_4(\eta)}D_4)\in\mathcal{W}_D:\ref{cond2w}\text{ holds}
 \} \ll N(\log N)^{-\frac{1}{4}+\epsilon}.\]
\end{lemma}

Combining Lemma~\ref{lemma:w1} and Lemma~\ref{lemma:w2} proves Theorem~\ref{thm:cond}.

\section{Contribution from non-generic $2$-Selmer elements} \label{section:nongen}

Take $\mathcal{G}_N$ to be the collection of $D\in\D_N$ that satisfies one of the following
\begin{enumerate}[label=(\textbf{P\arabic*})]
 \item $\omega(D)\geq 2\log\log N$, \label{EK}
 \item $D< \exp(3(\log N)^{1-2\epsilon})$, \label{size}
 \item at least one of the conditions~\ref{cond1} and~\ref{cond2} fails. \label{Sfail}
\end{enumerate}

\begin{lemma}\label{lemma:nongen}
We have
\[\sum_{D\in \mathcal{G}_N}\#\E_D(\Z)\ll N(\log N)^{-\frac{1}{4}+2\epsilon}.\]
\end{lemma}
\begin{proof}
By the Erd\H{o}s-{K}ac theorem~\cite{EdosKac}, the numbers of $D\in D_N$ from~\ref{EK} is bounded by
\[
\#\left\{D\in \D_N :\omega(D)\geq 2\log\log N\right\}\ll N(\log N)^{-\frac{1}{2}}.
\]
The number of $D\in\D_N$ that satisfies~\ref{size} is trivially bounded by
$\exp(3(\log N)^{1-\epsilon})$.
Theorem~\ref{thm:cond} bounds the number of $D\in\D_N$ that satisfies~\ref{Sfail}.
Therefore 
\[\#\mathcal{G}_N\ll N(\log N)^{-\frac{1}{4}+\epsilon}.\]

We now bound the contribution from the curves in $\mathcal{G}_N$ to the total number of integral points.
Applying~\eqref{eq:intmoments} with $k=1/\epsilon$, we have
\[
\sum_{D\in\D_N}(\#\E_D(\Z))^{\frac{1}{\epsilon}}\ll_{\epsilon} N.
\]
Using Hölder's inequality, 
\[
\sum_{D\in \mathcal{G}_N}\#\E_D(\Z)
\leq
\Big(\sum_{D\in \D_N}(\#\E_D(\Z))^{\frac{1}{\epsilon}}\Big)^{\epsilon}
(\#\mathcal{G}_N)^{1-\epsilon}
\ll_{\epsilon} N^{\epsilon}(\#\mathcal{G}_N)^{1-\epsilon}.
\]
The claims follows from putting $\#\mathcal{G}_N\ll N(\log N)^{-\frac{1}{4}+\epsilon}$.
\end{proof}

\section{Counting generic points}\label{sec:count}
By Lemma~\ref{lemma:nongen}, we may exclude any $D\in\mathcal{G}_N$.
Any integral point in $\E_D(\Z)$ maps to $2$-Selmer element, and hence to $\mathcal{W}_D$ under the map
\[\E_D(\Z)\hookrightarrow\E_D(\Q)\twoheadrightarrow\E_D(\Q)/2\E_D(\Q)\hookrightarrow\Sel_2(\E_D)\twoheadrightarrow\mathcal{W}_D.\]
Recall that 
\[ \E_D^*(\Z)= \E_D(\Z)\setminus\{ (0,0),(\pm D,0)\}.\]
For the non-trivial integral points that has image of the type~\ref{cond2}, we have the following bound from~\cite[Theorem~1.4]{Cintegral} and the discussion after~\cite[Theorem~10.1]{Cintegral}.
\begin{lemma}\label{lemma:torspoint}
We have
\[\sum_{D\in\D_N}\sum_{T\in\{O,(0,0),(\pm D,0)\}}\#(\E_D^*(\Z)\cap (T+2\E_D(\Q)))\ll\sqrt{N}\log N.\]
\end{lemma}
Therefore it remains to handle the integral points that arise from~\ref{cond1}.
Define
\[\mathcal{Z}_N\coloneqq \bigcup_{D\in\D_N\setminus\mathcal{G}_N}\{P\in \E_D(\Z): P\notin 2\E_D(\Q)+\{O,(0,0),(\pm D,0)\}\}.
\]
Then the image of $(x,y)\in \mathcal{Z}_N$ corresponds to $(D_1,D_2,D_3,D_4)\in\mathcal{W}_D$ of the type~\ref{cond1}. By Lemma~\ref{lemma:identifyS2}, \[D/\gcd(x,D)\in\{D_1,D_2,D_3,D_4\}.\] This implies that we can find a prime factor $M$ of $D/\gcd(x,D)$ of size 
\begin{equation}\label{eq:Mrange}
\exp((\log N)^{2\epsilon})<M<\exp((\log N)^{1-2\epsilon}).
\end{equation}
Now $M$ divides $D$ but does not divide $x$.
Therefore we can carry out the transformation in Lemma~\ref{lemma:reddisc} with this $M$. 
For each $P\in\mathcal{Z}_N$, fix a choice of integers $M$ and $k$ such that $M$ satisfies~\eqref{eq:Mrange} and $k^2\equiv x(P)\bmod M$. Define
\[\Phi:\mathcal{Z}_N\rightarrow V/\GL_2(\Z)\] 
by
\[P\mapsto(P,M,k)\xmapsto{\Psi}(F,(1,0))\mapsto F,\]
where $\Psi$ is the map defined in~\ref{eq:Psi}.
For any $P\in \E_D(\Z)\cap S_N$, write 
\[\tilde{D}=\frac{D}{M},\]
so $\Delta(F)=(2\tilde{D})^6$ if $F=\Phi(P)$.
Since $D\geq \exp(3(\log N)^{1-2\epsilon})$ by~\ref{size} and $M$ is in the range~\eqref{eq:Mrange}, we have 
\begin{equation}\label{eq:tDrange}
 \exp(2(\log N)^{1-2\epsilon})<D\exp(-(\log N)^{1-2\epsilon})\leq \tilde{D}<D\exp(-(\log N)^{2\epsilon}).
\end{equation}

\subsection{Points lowered to the same quartic}
 We now show that each binary quartic form in the image of $\Phi$ cannot arise from too many integral points.
\begin{lemma}
For any $F\in\im\Phi$, we have
\[\#\Phi^{-1}(F)\ll 1,\]
where the implied constant is independent of $F$.
\end{lemma}

\begin{proof}

Lemma~\ref{lemma:injective} implies that $\Psi$ is injective.
Fix some $F_0\in\im \Phi$. Suppose $(F,(1,0))\in\im\Psi$ is such that $F$ and $F_0$ are $\GL_2(\Z)$-equivalent, so we can write 
\[F_0(X,Y)=F((X,Y)\cdot \gamma)\] for some $\gamma\in\GL_2(\Z)$. Then 
\[\gamma\cdot (F(X,Y),(1,0))=(F((X,Y)\cdot \gamma),(1,0)\cdot \gamma^{-1})
=(F_0(X,Y),(1,0)\cdot \gamma^{-1}).\]
This gives a solution $(x,y)=(1,0)\cdot \gamma^{-1}$ to the Thue inequality 
\begin{equation}\label{eq:Thueint}
 1 \leq |F_0(x,y)|\leq h,
\end{equation}
where
$h\coloneqq\exp((\log N)^{1-2\epsilon})$ is taken so that $h$ is greater than any $M$ in~\eqref{eq:Mrange}. 
In particular this solution is primitive ($x$ and $y$ coprime), since $\gamma^{-1}\in\GL_2(\Z)$ has determinant $\pm 1$ and entries in $\Z$.
The solutions to the Thue inequality constructed from different elements $(F,(1,0))\in\im\Psi$ are distinct as long as $(F,(1,0))$ are from different $\GL_2(\Z)$-equivalence classes. Indeed, suppose $F_0(X,Y)=F_1((X,Y)\cdot \gamma_1)$ and $F_0(X,Y)=F_2((X,Y)\cdot \gamma_2)$, then if the solutions produced on~\eqref{eq:Thueint} are same, we also have $(1,0)\cdot \gamma_1^{-1}=(1,0)\cdot \gamma_2^{-1}$, so $\gamma_1\cdot (F_1(X,Y),(1,0))=\gamma_2\cdot (F_2(X,Y),(1,0))$.

A result by Evertse 
\cite[Theorem~6.3]{Evertse} implies that when $2^8\Delta(F_0)\geq (13h)^{10}$, the number of solutions to~\eqref{eq:Thueint} is bounded by some absolute constant. Since $\Delta(F_0)=(2\tilde{D})^6\gg \exp(12(\log N)^{1-2\epsilon})$ from~\eqref{eq:tDrange}, and $h^{10}=\exp(10(\log N)^{1-2\epsilon})$, we conclude that the number of possible classes $(F,(1,0))$ associated to each class of $F_0$ is absolutely bounded.
\end{proof}

\subsection{Integral points with bounded height}

Every integral binary quartic form is $\SL_2(\Z)$-equivalent to at least one reduced form~\cite{Cremona}. The seminvariant $a,H$ of the reduced form are bounded in terms of $I$ and $J$. We restate a theorem in~\cite{Cremona} in terms of our rescaled seminvariants. The scale factors of the seminvariants can be found in Section~\ref{sec:quartics}.

\begin{theorem}[{\cite[Proposition 11]{Cremona}}]\label{theorem:rangered}
Suppose $F_0(X,Y)\in\Z[X,Y]$ is a $\GL_2(\Z)$-reduced quartic form, and $\Delta(F_0)>0$, with leading coefficient $a$ and seminvariant $H$.
Order the three real roots $\phi_1,\phi_2,\phi_3$ of $X^3-\frac{I}{4}X-\frac{J}{4}$ so that
$a\phi_1<a\phi_2<a\phi_3$.
Then $(a,H)$ satisfies one of the following:
\begin{enumerate}
\item $|a|\leq \frac{4}{3}|\phi_1-\phi_3|$ and $\max\{a\phi_1,\ a\phi_3-4\phi_3^2+\frac{1}{3}I\}\leq H\leq a\phi_2$; or
\item $ |a|\leq \frac{4}{3}|\phi_1-\phi_2|$ and 
 $a\phi_3\leq H\leq a\phi_2-4\phi_2^2+\frac{1}{3}I$.
\end{enumerate}

\end{theorem}

For $\im\Phi$, $\Delta(F)=(2\tilde{D})^6>0$. Also $I(F)=4\tilde{D}^2$ and $J(F)=0$, so $\{\phi_1,\phi_3\}=\{-\tilde{D},\tilde{D}\}$ and $\phi_2=0$.
Suppose $F_0$ is a reduced form of $F$ with leading coefficient $a$ and seminvariant $H$. Then the two possible cases in Lemma~\ref{theorem:rangered} both lead to 
\begin{equation}\label{eq:aHbounds}
 |a|\leq \frac{8}{3}\tilde{D}\qquad\text{ and }\qquad|H|\leq \frac{4}{3}\tilde{D}^2.
\end{equation}
The syzygy in~\eqref{eq:syzygy} for $F_0$ now takes the form
\[H^3-(a\tilde{D})^2H=\left(\frac{1}{2}R\right)^2.\]
Notice that this gives an integral point $(H,\frac{1}{2}R)\in \E_{|a\tilde{D}|}(\Z)$ when $a\neq 0$. Below we show that the possibility that $a= 0$ does not happen because we restricted our counting to $\mathcal{Z}_N$.
\begin{lemma}\label{lemma:anonzero}
Suppose $F\in\im\Phi$. Then any form in the $\SL_2(\Z)$-equivalence class of $F$ has non-zero leading coefficient. 
\end{lemma}
\begin{proof}
Assume for contradiction that $\Phi(P)=F$ for some $P=(c,d)\in\mathcal{Z}_N$ and $F$ is equivalent to some quartic form with leading coefficient $0$. Then there is a non-trivial integral solution to $F(X,Y)=0$.
From $\Phi(P)=F$, we know that $F(X,Y)=\frac{1}{M^3}f_P(MX+kY,Y)$ for some $M,k\in\Z$, so $f_P(X,Y)=0$ also has a non-trivial solution, say $(x_0,y_0)$.
Then from the expression of $f_P$ in~\eqref{eq:fpdef},
\[
f_P(x_0,y_0)=x_0^4-6cx_0^2y_0^2+8dx_0y_0^3+(4D^2-3c^2)y_0^4=0.
\]
We see that $y_0\neq 0$ since the solution is non-trivial.
The roots of $f_P(X,1)$ are
\begin{align*}
 \frac{x_0}{y_0}& =-\sqrt{c}+\sqrt{c+D}+\sqrt{c-D},\quad
-\sqrt{c}-\sqrt{c+D}-\sqrt{c-D}, \\
&\quad\qquad \sqrt{c}+\sqrt{c+D}-\sqrt{c-D},\quad
\sqrt{c}-\sqrt{c+D}+\sqrt{c-D}.
\end{align*}
For $x_0/y_0$ to be rational, it must be that $\theta(P)=(1,1,1)$, where $\theta$ is the $2$-descent homomorphism defined in~\eqref{eq:theta}. This implies that $P\in 2\E_D(\Q)$, but such points are not in $\mathcal{Z}_N$.
\end{proof}

The $\SL_2(\Z)$-equivalence class of $F_0$ is determined by $(a,\tilde{D},H,R)$ and $R$ is fixed by $(a,\tilde{D},H)$ up to $\pm $ sign, so it suffices to count the number of $(a,\tilde{D},H)$ that arise from this reduction, with the bounds~\eqref{eq:aHbounds}.

\subsection{Torsion points}
We have reduced the problem to counting integral points with bounded height on $\E_{|a\tilde{D}|}$, but since there are $2$-torsion points on every curve, we need to deal with this possibility separately.

Here we bound those $F\in\im\Phi$ that reduces to a form whose syzygy~\eqref{eq:syzygy} produces a torsion point on some $\E_{|a\tilde{D}|}$.
From~\eqref{eq:tDrange},
\[\tilde{D}\leq N\exp(-(\log N)^{2\epsilon}).\]
\begin{lemma}\label{lemma:removetorsion}
Let $\tilde{N}=N\exp(-(\log N)^{2\epsilon})$.
The total number of $\SL_2(\Z)$-equivalence classes of integer-matrix binary forms $F$ that satisfy $H(F)\in\{-a(F)\tilde{D},0,a(F)\tilde{D}\}$, $I(F)=(2\tilde{D})^2$ and $J(F)=0$ for some $\tilde{D}\in\D_{\tilde{N}}$,
is bounded by
\[\ll N\exp(-(\log N)^{\epsilon}).\]
\end{lemma}
\begin{proof}
Suppose that 
$F(X,Y)=a_0X^4+4a_1X^3Y+6a_2X^2Y^2+4a_3XY^3+a_4Y^4$.
Since $H(F)=a_1^2-a_0a_2$, if $H(F)\in\{-a_0\tilde{D},0,a_0\tilde{D}\}$, it must be that $a_0\mid a_1^2$.
Then $a_0\mid \Delta(F)=(2\tilde{D})^6$ by the discriminant formula. Therefore for each $\tilde{D}$, there can only be a maximum of $7^{\omega(2\tilde{D})}$ possible $a_0$.
Sum over $\tilde{D}\in\D_{\tilde{N}}$, and apply Lemma~\ref{lemma:Shiu}, we can bounde the number of classes 
\[\ll \sum_{\tilde{D}\leq \tilde{N}}7^{\omega(\tilde{D})}\ll \tilde{N}(\log \tilde{N})^{6}.\]
Put in $\tilde{N}=N\exp(-(\log N)^{2\epsilon})$ proves Lemma~\ref{lemma:removetorsion}.
\end{proof}

\subsection{Non-torsion points}
We now bound those $F\in\im\Phi$ that reduces to a form whose syzygy~\eqref{eq:syzygy} produces a non-torsion point on some $\E_{|a\tilde{D}|}$ with~\eqref{eq:aHbounds}.
Since the $\E_D$ that satisfies~\ref{EK} have been removed, we can assume that $\omega(\tilde{D})<\omega(D)<2\log\log N$.
Also by Lemma~\ref{lemma:anonzero}, $a\neq 0$.

\begin{lemma}\label{lemma:bdheight}
Let $\tilde{N}=N\exp(-(\log N)^{2\epsilon})$. Then
\[ \sum_{1\leq|a|\leq \frac{8}{3}\tilde{N}}\sum_{\substack{\tilde{D}\in\D_{\tilde{N}}\\\omega(\tilde{D})<2\log\log N}}\#\left\{\left(H,\frac{1}{2}R\right)\in \E_{|a\tilde{D}|}^*(\Z):|H|\leq \frac{4}{3}\tilde{D}^2\right\}\ll N\exp(-(\log N)^{\epsilon}) .\]
\end{lemma}

\begin{proof}
Write $n=|a\tilde{D}|\leq \frac{8}{3}\tilde{N}^2$. 
For each positive integer $n$, the number of ways to factorise $n$ into a product of $a$ and $\tilde{D}$ such that $\tilde{D}$ is squarefree and $\omega(\tilde{D})<2\log\log N$, is bounded by
\begin{equation}\label{eq:factor}
\sum_{k\leq\min\{\omega(n),2 \log\log N\}}\binom{\omega(n)}{k}<
\sum_{k\leq 2\log\log N} \left(\omega(n)\right)^k
\ll e^{3(\log \log N)^2},
\end{equation} 
where we have used the fact that $\omega(n)\ll \log N$.

The number of integral points we are counting are of bounded height $|x(P)|\leq \frac{4}{3}\tilde{N}^2$, so
applying a result by Le Boudec~\cite[Theorem~2]{LeBoudec} we get
\begin{equation}\label{eq:numberofpts}
\sum_{n\geq 1}\#\left\{P\in \E_{n}^*(\Z): |x(P)|\leq \frac{4}{3}\tilde{N}^2\right\} \ll \tilde{N}(\log \tilde{N})^6.
\end{equation}
Now multiplying~\eqref{eq:factor} and~\eqref{eq:numberofpts}, we get that the total number of triples $(H,a,\tilde{D})$ is \[\ll N\exp(-(\log N)^{2\epsilon}+4(\log \log N)^2) .\]
This proves the claim.
\end{proof}

Theorem~\ref{theorem:main} follows from Lemma~\ref{lemma:nongen}, Lemma~\ref{lemma:torspoint}, Lemma~\ref{lemma:removetorsion} and Lemma~\ref{lemma:bdheight}.

\end{document}